\documentclass{article}

\usepackage{amsmath,amsthm,amssymb,amsfonts}
\usepackage{verbatim}

\usepackage{stmaryrd} 
\usepackage{hyperref}

\usepackage[colorinlistoftodos]{todonotes}

\usepackage{tikz}
\usepackage{tkz-berge, tkz-graph}
\tikzset{vertex/.style={circle,draw,fill,inner sep=0pt,minimum size=1mm}}

\theoremstyle{plain}
\newtheorem{thm}{Theorem}

\newtheorem{prop}[thm]{Proposition}
\newtheorem{cor}[thm]{Corollary}

\newtheorem{remark}[thm]{Remark}

\theoremstyle{definition}
\newtheorem{definition}[thm]{Definition}
\newtheorem{exl}[thm]{Example}

\numberwithin{thm}{section}

\newcommand{\adj}{\leftrightarrow}

\DeclareMathOperator{\id}{id}
\DeclareMathOperator{\MF}{MF}
\DeclareMathOperator{\XF}{XF}
\DeclareMathOperator{\ind}{ind}
\def\Z{{\mathbb Z}}
\def\N{{\mathbb N}}
\def\R{{\mathbb R}}

\begin{document}
\title{Remarks on Fixed Point Assertions in Digital Topology}
\author{Laurence Boxer
         \thanks{
    Department of Computer and Information Sciences,
    Niagara University,
    Niagara University, NY 14109, USA;
    and Department of Computer Science and Engineering,
    State University of New York at Buffalo.
    E-mail: boxer@niagara.edu
    }
    \and{
    P. Christopher Staecker
    \thanks{Department of Mathematics, Fairfield University, Fairfield, CT 06823-5195, USA.
    E-mail: cstaecker@fairfield.edu
    }
    }
}
\date{ }
\maketitle

\begin{abstract}
Several recent papers in digital topology have sought to
obtain fixed point results by mimicking the use of tools
from classical topology, such as complete metric spaces
and homotopy invariant fixed point theory. We show that
in many cases, researchers using these tools have derived
conclusions that are incorrect or trivial.

\end{abstract}

\section{Introduction}
Recent papers have attempted to apply to digital images
ideas from Euclidean topology and real analysis concerning 
metrics and fixed points. While the underlying motivation
of digital topology comes from Euclidean topology and real
analysis, some applications recently featured in the 
literature seem of doubtful worth. Although papers
including~\cite{Rosenfeld87,Bx-etal} have valid and
interesting results for fixed points and for ``almost" or
``approximate" fixed points in digital topology,
many other published assertions concerning fixed points in 
digital topology are incorrect or trivial (e.g., applicable
only to singletons, or functions studied forced to be 
constant), as we will discuss in the current paper.

\section{Preliminaries}
We let $\Z$ denote the set of integers, and $\R$, the real line.

We consider a digital image as a graph $(X,\kappa)$, where 
$X \subset \Z^n$ for some positive integer $n$ and
$\kappa$ is an adjacency relation on $X$. 

A digital metric space is~\cite{EgeKaraca} a triple
$(X,d,\kappa)$ where $(X,\kappa)$ is a digital image and
$d$ is a metric for $X$. In~\cite{EgeKaraca}, $d$ was taken
to be the Euclidean metric, but we will not limit our
discussion to the Euclidean metric.

The {\em diameter} of a metric space $(X,d)$
is 
\[diam \, X = \max \{d(x,y) \, | \, x,y \in X\}.\]

\subsection{Adjacencies}
The most commonly used adjacencies for digital images
are the $c_u$-adjacencies, defined as follows.
\begin{definition}
Let $p,q \in \Z^n$, $p=(p_1,\ldots,p_n)$, 
$q=(q_1,\ldots,q_n)$, $p \ne q$. Let $1 \le u \le n$. We say
$p$ and $q$ are $c_u$-adjacent, denoted
$p \adj_{c_u} q$ or $p \adj q$ when the adjacency is
understood, if
\begin{itemize}
    \item for at most $u$ distinct indices $i$,
          $|p_i - q_i| = 1$, and
    \item for all other indices $j$, $p_j = q_j$.
\end{itemize}
\end{definition}

Often, a $c_u$-adjacency is denoted by the number
of points in $\Z^n$ that are $c_u$-adjacent to a given
point. E.g.,
\begin{itemize}
    \item in $\Z^1$, $c_1$-adjacency is 2-adjacency;
    \item in $\Z^2$, $c_1$-adjacency is 4-adjacency and
          $c_2$-adjacency is 8-adjacency;
    \item in $\Z^3$, $c_1$-adjacency is 8-adjacency,
          $c_2$-adjacency is 18-adjacency, and
          $c_3$-adjacency is 26-adjacency.
\end{itemize}

An adjacency often used for Cartesian products of
digital images is the
{\em normal product adjacency}, denoted in the following by
$\kappa_*$ and defined~\cite{Berge}
as follows. Given digital images $(X,\kappa)$ and
$(Y,\lambda)$ and points $x, x' \in X$, $y,y' \in Y$,
we have $(x,y) \adj_{\kappa_*} (x',y')$ in $X \times Y$
if and only if one of the following holds.
\begin{itemize}
    \item $x \adj_{\kappa} x'$ and $y=y'$, or
    \item $x = x'$ and $y \adj_{\lambda} y'$, or
    \item $x \adj_{\kappa} x'$ and $y \adj_{\lambda} y'$.
\end{itemize}

Other adjacencies for digital images are discussed in
papers such as \cite{Herman,Bx17,Bx18}.

A {\em digital interval} is a digital image of the form
$([a,b]_Z, 2)$, where $a < b$ and 
$[a,b]_Z = \{z \in \Z \, | \, a \le z \le b\}$.

\subsection{$\ell_p$ metric}
Let $X \subset \R^n$ and let $x=(x_1,\ldots, x_n)$ and
$y=(y_1,\ldots,y_n)$ be points of $X$. Let
$1 \le p \le \infty$. The $\ell_p$ metric $d$ for $X$ is defined by
\[ d(x,y) = \left \{ \begin{array}{ll}
        \left (    \sum_{i=1}^n |x_i - y_i|^p \right )^{1/p} &
            \mbox{for } 1 \le p < \infty; \\
            \max \{|x_i - y_i|\}_{i=1}^n &
            \mbox{for } p = \infty. \end{array}
\right .
\]
For $p=1$, this gives us the {\em Manhattan metric}
$d(x,y) = \sum_{i=1}^n |x_i - y_i|$; for $p=2$, we have the
{\em Euclidean metric} $d(x,y) = (\sum_{i=1}^n |x_i - y_i|^2)^{1/2}$.

Notice that for any $\ell_p$ metric $d$, if $x,y \in \Z^n$ and
$d(x,y) < 1$, then $x=y$.

\subsection{Digital continuity and homotopy}
\begin{definition}
\label{digitalContinuity}
\cite{Rosenfeld87,Bx99}
A function $f: (X,\kappa) \to (Y, \lambda)$ between
digital images is $(\kappa,\lambda)$-{\em continuous}
(or just {\em continuous} when $\kappa$ and $\lambda$ are
understood) if for every $\kappa$-connected subset 
$X'$ of $X$, $f(X')$ is a $\lambda$-connected subset of $Y$.
\end{definition}

\begin{thm}
{\rm \cite{Bx99}}
A function $f: (X,\kappa) \to (Y, \lambda)$ between
digital images is $(\kappa,\lambda)$-{\em continuous} if
and only if $x \adj_{\kappa} x'$ in $X$ implies either
$f(x)=f(x')$ or $f(x) \adj_{\lambda} f(x')$ in $Y$.
\end{thm}

As in topology, the digital topology notion of homotopy
can be understood as one function deforming in a continuous
fashion into another. Precisely, we have the following.

\begin{definition}
\cite{Bx99}
Let $f,g: (X,\kappa) \to (Y,\lambda)$. We say
$f$ and $g$ are homotopic, denoted 
$f \simeq_{(\kappa,\lambda)} g$ or $f \simeq g$ when
$\kappa$ and $\lambda$ are understood, if there is
a function $F: X \times [0,m]_{\Z} \to Y$ for some 
$m \in \N$ such that
\begin{itemize}
    \item $F(x,0)=f(x)$ and $F(x,m)=g(x)$ for all $x \in X$.
    \item The induced function $F_t: X \to Y$ defined by
          $F_t(x)=F(x,t)$ is $(\kappa,\lambda)$-continuous
          for all $t \in [0,m]_{\Z}$.
    \item The induced function $F_x: [0,m]_{\Z} \to Y$
          defined by $F_x(t)=F(x,t)$ is 
          $(2,\lambda)$-continuous for all $x \in X$.
\end{itemize}
\end{definition}

\subsection{Cauchy sequences and complete metric spaces}
The papers~\cite{EgeKaraca,Han16,Hossain-etal,Jain,JR17,JR18,Mishra-etal,Ege-etal} apply to 
digital images the notions of Cauchy 
sequence and complete metric space.
Since for common metrics such as an $\ell_p$ metric, a digital metric space is discrete, the digital versions of 
these notions are quite limited.

Recall that a sequence of points $\{x_n\}$ in a metric space 
$(X,d)$ is a {\em Cauchy sequence}
if for all $\varepsilon > 0$ there exists $n_0 \in \N$
such that $m,n > n_0$ implies
$d(x_m,x_n) < \varepsilon$. If every Cauchy sequence in $X$ has a limit, then $(X,d)$ is a 
{\em complete metric space}.

It has been shown that under a mild additional assumption, a digital
Cauchy sequence is eventually constant. The following is an easy 
generalization of Proposition~3.6 of~\cite{Han16}, where only
the Euclidean metric was considered. The proof given in~\cite{Han16}
is easily modified to give the following.

\begin{thm}
\label{Han-Cauchy}
Let $a > 0$. If $d$ is a metric on a digital image $(X,\kappa)$ such that
for all distinct $x,y \in X$ we have $d(x,y) > a$, then
for any Cauchy sequence $\{x_i\}_{i=1}^{\infty} \subset X$ 
there exists $n_0 \in \N$ such that $m,n > n_0$ implies $x_m = x_n$.
\end{thm}

An immediate consequence of Theorem~\ref{Han-Cauchy} is the 
following.

\begin{cor}
{\rm \cite{Han16}}
\label{Han-Cauchy-cor}
Let $(X,d,\kappa)$ be a digital metric space. If $d$ is a metric on $(X,\kappa)$ such 
that for all distinct $x,y \in X$ we have $d(x,y) > a$ for some constant $a > 0$, 
then any Cauchy sequence in $X$ is eventually constant, and $(X,d)$ is a complete metric space.
\end{cor}

\begin{remark}
{\rm
It is easily seen that the hypotheses of 
Theorem~\ref{Han-Cauchy} and 
Corollary~\ref{Han-Cauchy-cor} 
are satisfied for any finite digital metric space, or for a 
digital metric space $(X,d,\kappa)$ for which the metric $d$ 
is any $\ell_p$ metric. Thus, a Cauchy sequence that is
not eventually constant can only occur in an infinite
digital metric space with an unusual metric. Such an
example is given below.
}
\end{remark}

\begin{exl}
Let $d$ be the metric on $(\N,c_1)$ defined by
$d(i,j) = |1/i - 1/j|$. Then $\{i\}_{i=1}^{\infty}$ is
a Cauchy sequence for this metric that does not have a
limit.
\end{exl}

\subsection{Digital fixed points and approximate fixed points}
The study of the fixed points of continuous self-maps
is prominent in many areas of mathematics. We say
a topological space or a digital image $X$ has the 
{\em fixed point property} if every continuous (respectively,
digitally continuous)
$f: X \to X$ has a fixed point, i.e., a point $p \in X$ such
that $f(p)=p$. 

A version of Theorem~\ref{singleton} below was proved by Rosenfeld in \cite{Rosenfeld87} for the case when $X$ is a \emph{digital picture}, that is, a digital image of the form
$\Pi_{i=1}^n [a_i,b_i]_{\Z} \subset \Z^n$ with the 
$c_n$-adjacency. For general digital images, 
Theorem \ref{singleton} was proved in \cite{Bx-etal}.
 
\begin{thm}
\label{singleton}
A digital image $(X,\kappa)$ has the fixed point property
if and only if $X$ is a singleton.
\end{thm}

This theorem led to the study in~\cite{Bx-etal}
of the {\em approximate fixed point property}, an idea 
suggested by results of~\cite{Rosenfeld87}. An
{\em approximate fixed point}~\cite{Bx-etal}, called an
{\em almost fixed point} in~\cite{Rosenfeld87}, of a
continuous function $f: (X,\kappa) \to (X,\kappa)$ is a point $p \in X$
such that $f(p)=p$ or $f(p) \adj_{\kappa} p$.
A digital image $(X,\kappa)$ has the {\em approximate
fixed point property (AFPP)}~\cite{Bx-etal} if for
every continuous $f: X \to X$ there is an approximate
fixed point of $f$.

We have rephrased the following to conform 
with terminology used in this paper.

\begin{thm}
{\rm (Theorem~4.1 of~{\rm 
\cite{Rosenfeld87}})}
\label{4.1Rose}
Every digital picture 
$(\Pi_{i=1}^n [a_i,b_i]_{\Z}, c_n)$ has the AFPP.
\end{thm}

In Remark 6.2 (2) of~\cite{Han17}, the author 
incorrectly attributes to~\cite{Rosenfeld87} the claim that 
``Every digital image $(Y, 8)$ has the AFPP." The
attribution is incorrect, since, as we have shown above,
the citation should be about digital pictures, not the
more general digital images. Further, the claim is false, as the following example shows.

\begin{exl}
Let $n \ge 4$ and let
$Y = (\{y_i\}_{i=0}^{n-1},8) \subset \Z^2$ be a digital
simple closed curve with the points $y_i$ indexed
circularly. Then $(Y,8)$ does not have the AFPP.
\end{exl}

\begin{proof}
The function $f: Y \to Y$ defined by
$f(y_i) = y_{(i+2) \mod n}$ is easily seen to be
$(8,8)$-continuous and free of approximate fixed points.
\end{proof}

\subsection{Contraction and expansion functions}
\label{contractionFunctions}
We introduce several fixed point theorems for digital topology that
are modeled on analogs for the topology and analysis of $\R^n$.

In the following definitions, we assume $(X,d,\kappa)$ is a digital 
metric space and $f: X \to X$ is a function.

\begin{definition}
\label{contractionDef}
{\rm \cite{EgeKaraca}}
If for some $\alpha \in (0,1)$ and all  
          $x,y \in X$, $d(f(x),f(y)) < \alpha d(x,y)$, then
          $f$ is a {\em digital contraction map}. We say
$\alpha$ is the {\em multiplier}.
\end{definition}         
         
Note such a function should not be confused with a
{\em digital contraction}~\cite{Bx94}, a homotopy between 
an identity map and a constant function.
    
\begin{definition}
\label{KannanDef}
If         \[ d(f(x),f(y)) \le \alpha [d(x, f(x)) + d(y, f(y))] \]
          for all $x,y \in X$, where $0 < \alpha < 1/2$, we say
          $f$ is a {\em Kannan contraction map}.
\end{definition}
  
\begin{definition}
\label{ChatterjeaDef}
    If
          \[ d(f(x),f(y)) \le \alpha [d(x, f(y)) + d(y, f(x))] \]
          for all $x, y \in X$, where $0 < \alpha < 1/2$, we say
          $f$ is a {\em Chatterjea contraction map}.
    \end{definition}
    
\begin{definition}
\label{ReichDef}
If
\[ d(f(x),f(y)) \le ad(x, f(x)) + bd(y, f(y)) + cd(x, y)
\]
          for all $x,y \in X$ and all nonnegative $a,b,c$ such that
          $a+b+c < 1$, then $f$ is a {\em Reich contraction map}.
\end{definition} 

\begin{prop}
\label{ReichIsContraction}
A Reich contraction map is a digital contraction map and is a
Kannan contraction map.
\end{prop}

\begin{proof}
\label{ReichParticularizes}
Let $f$ be a Reich contraction map. 
That $f$ is a digital contraction map follows from the observation 
of~\cite{Reich} that in Definition~\ref{ReichDef}, we can take 
$a=b=0$ and obtain the conclusion from 
Definition~\ref{contractionDef}. That $f$ is a Kannan contraction 
map follows from the observation that in Definition~\ref{ReichDef}, 
we can take $a=b \in (0,1/2)$ and $c=0$
to obtain the conclusion from Definition~\ref{KannanDef}.
\end{proof}

\begin{definition}
\label{Zamf}
\cite{Mishra-etal}
Let $(X,d,\kappa)$ be a digital metric space and let
$f: X \to X$ be a function. If there exists
$\alpha \in (0,1)$ such that for all $x,y \in X$ we have
\[ d(f(x),f(y)) \le \alpha \, \max\{d(x,y), 
   \frac{d(x,f(x)) + d(y,f(y))}{2},
   \frac{d(x,f(y)) + d(y,f(x))}{2}\}
\]
then $f$ is called a {\em Zamfirescu digital contraction}.
\end{definition}

\begin{definition}
\label{Rhoades}
\cite{Mishra-etal}
Let $(X,d,\kappa)$ be a digital metric space and let
$f: X \to X$ be a function. If there exists
$\alpha \in (0,1)$ such that for all $x,y \in X$ we have
\[ d(f(x),f(y)) \le \alpha \, \max\{d(x,y), 
   \frac{d(x,f(x)) + d(y,f(y))}{2},
   d(x,f(x)), d(y,f(y))\}
\]
then $f$ is called a {\em Rhoades digital contraction}.
\end{definition}

\begin{prop}
\label{ZamfAndRhoadesContractions}
Let $(X,d,\kappa)$ be a digital metric space and let
$f: X \to X$ be a function. If $f$ is a Zamfirescu digital 
contraction or a Rhoades digital contraction, then $f$
is a digital contraction map.
\end{prop}

\begin{proof}
From Definitions~\ref{Zamf} and \ref{Rhoades},
$d(f(x), f(y)) \le \alpha d(x,y)$ for all $x,y \in X$. The
assertion follows from Definition~\ref{contractionDef}.
\end{proof}

The following is a minor generalization of a definition
in~\cite{Hossain-etal}. Therefore, results we derive in this 
paper for digitally $(\alpha, \kappa)$-uniformly locally
contractive functions apply to the version
in~\cite{Hossain-etal}.

\begin{definition}
\label{locallyContractive}
Suppose $0 \le \alpha < 1$. Let $(X,d,\kappa)$ be a digital 
metric space. Let $f: X \to X$ be a function such that
$d(x,y) \le 1$ implies $d(f(x),f(y)) \le \alpha d(x,y)$. 
Then $f$ is called 
{\em digitally $(\alpha, \kappa)$-uniformly locally
contractive}.
\end{definition}

\begin{prop}
A digital contraction map with multiplier $\alpha$ is a 
digitally $(\alpha, \kappa)$-uniformly locally 
contractive map.
\end{prop}

\begin{proof}
This is obvious from Definition~\ref{contractionDef} and
Definition~\ref{locallyContractive}.
\end{proof}

Below, we define a set of functions $\Psi$ that will
be used in the following.

\begin{definition}
{\rm \cite{JR18}}
\label{psiFamily}
Let $\Psi$ be a set of functions 
$\psi: [0,\infty) \to [0,\infty)$ such that for each
$\psi \in \Psi$ we have
\begin{itemize}
    \item $\psi$ is nondecreasing, and
    \item there exists $k_0 \in \N$, $a \in (0,1)$, and a
          convergent series $\sum_{k=1}^{\infty} v_k$ of
          non-negative terms such that $k \ge k_0$ implies
          $\psi^{k+1}(t) \le a \psi^k(t) + v_k$ for all
          $t \in [0,\infty)$, where $\psi^k$ represents the
          $k$-fold composition of $\psi$.
\end{itemize}
\end{definition}

The following will be used later in the paper.

\begin{exl}
\label{0inPsi}
The constant function with value 0 is a member of~$\Psi$.
\end{exl}

\begin{definition}
{\rm \cite{SametEtAl}}
\label{alphaAdmissible}
Let $T: X \to X$ and $\alpha: X \times X \to [0, \infty)$.
We say $T$ is {\em $\alpha$-admissible} if
$\alpha(x,y) \ge 1$ implies $\alpha(T(x),T(y)) \ge 1$.
\end{definition}

\begin{definition}
{\rm \cite{SametEtAl}}
\label{alpha-psi-contractive}
Let $(X,d)$ be a metric space, $T: X \to X$, 
$\alpha: X \to X$, and $\psi \in \Psi$. We say $T$ is an
{\em $\alpha-\psi$-contractive mapping} if
$\alpha(x,y) d(T(x),T(y)) \le \psi(d(x,y))$
for all $x,y \in X$.
\end{definition}

\begin{remark}
{\rm \cite{JR18}}
\label{contractionIsAlphaPsiContractive}
A digital contraction map $f: (X,d,\kappa) \to (X,d,\kappa)$
is an $\alpha-\psi$-contractive mapping for $\alpha(x,y)=1$
and $\psi(t) = \lambda t$ for $\lambda \in (0,1)$.
\end{remark}

\begin{definition}
\label{expansive3}
{\rm \cite{JR18}}
Let $(X,d,\kappa)$ be a digital metric space, $T: X \to X$,
$\beta: X \times X \to [0,\infty)$, and 
$\psi, \phi \in \Psi$ such that
\[ \psi(d(T(x), T(y))) \ge \beta(x,y) \psi(d(x,y)) +
                           \phi(d(x,y))
\]
for all $x,y \in X$. Then $T$ is a
{\em $\beta - \psi - \phi$-expansive mapping}.
\end{definition}

Depending on the choice of functions $\beta, \phi$ in 
Definition~\ref{expansive3}, the definition may not be 
very discriminating, as we see in the following.

\begin{prop}
\label{expansiveNondistinguishing}
Every function $T: X \to X$ is a
$\beta - \psi - \phi$-expansive mapping
if we take $\beta$ and $\phi$ to be
constant functions with value 0.
\end{prop}

\begin{proof}
The assertion follows from Example~\ref{0inPsi} and
Definition~\ref{expansive3}.
\end{proof}

\begin{definition}
{\rm \cite{Dol-Nal}}
\label{weaklyUnif}
Let $(X,d,\kappa)$ be a digital metric space.
Let $T: X \to X$. Then $T$ is a 
{\em weakly uniformly strict digital contraction} if given
$\varepsilon > 0$ there exists $\delta > 0$ such that
$\varepsilon < d(x,y) < \varepsilon + \delta$ implies
$d(T(x),T(y)) < \varepsilon$ for all $x,y \in X$.
\end{definition}

\begin{definition}
{\rm \cite{JR18}}
\label{expansive}
Let $(X,d,\kappa)$ be a complete digital metric space.
Let $T: X \to X$. If $T$ satisfies the condition 
$d(T(x),T(y)) \ge k d(x,y)$ for all $x,y \in X$ and
some $k >1$, then $T$ is a {\em digital expansive mapping}.
\end{definition}

\begin{exl}
\label{expansiveExl}
The function $T: \N \to \N$ defined by $T(n)=2n$ is a
digital expansive mapping, using the usual Euclidean metric.
This map is not $c_1$-continuous~\cite{Rosenfeld87}.
\end{exl}

The literature contains the following theorems concerning 
fixed points for such functions.

The following is a digital 
version of the Banach contraction principle~\cite{Banach}.

\begin{thm}
\label{Banach}
{\rm \cite{EgeKaraca}}
Let $(X,d,\kappa)$ be a complete digital metric space,
where $d$ is the Euclidean metric in $\Z^n$. Let $f: X \to X$ be a digital
contraction map. Then $f$ has a unique fixed point.
\end{thm}

The following is a digital version 
of the Kannan fixed point theorem~\cite{Kannan}.

\begin{thm}
{\rm \cite{Ege-etal}}
\label{Kannan}
Let $f: X \to X$ be a Kannan contraction map on a digital metric space
$(X,d,\kappa)$. Then $f$ has a unique fixed point in $X$.
\end{thm}

The following is a digital version of the
Chatterjea fixed point theorem~\cite{Chatterjea}.

\begin{thm}
{\rm \cite{Ege-etal}}
\label{Chatterjea}
If $f: X \to X$ is a Chatterjea contraction map on a digital metric
space $(X,d,\kappa)$, then $f$ has a unique fixed point.
\end{thm}

The following gives digital
versions of Zamfirescu~\cite{Zamf72} and 
Rhoades~\cite{Rhoades} fixed point theorems.

\begin{thm}
{\rm \cite{Mishra-etal}}
\label{ZamfAndRhoadsFixed}
Let $d$ be the Euclidean metric on $\Z^n$ and let
$(X,d,\kappa)$ be a digital metric space.
If $f: X \to X$ is a Zamfirescu digital contraction or a 
Rhoades contraction map, then $f$ has a unique fixed point.
\end{thm}

The following is a digital version of the Reich fixed point 
theorem~\cite{Reich}. We give a simpler proof of the 
digital version than appeared in~{\rm \cite{Ege-etal}}.

\begin{thm}
\label{Reich}
Let $f: (X,d,\kappa) \to (X,d,\kappa)$ be a Reich 
contraction map on a digital metric space. Then $f$ has a 
unique fixed point in $X$.
\end{thm}

\begin{proof}
This follows immediately from 
Proposition~\ref{ReichIsContraction} and 
Theorem~\ref{Banach}.
\end{proof}

The following is a version of the Edelstein fixed point 
theorem~\cite{Edelstein} for digital images.

\begin{thm}
\label{HossainFixed}
{\rm \cite{Hossain-etal}}
A digitally $(\alpha, \kappa)$-uniformly locally 
contractive function on a connected complete digital
metric space has a unique fixed point.
\end{thm}

\section{Digital homotopy fixed point theory}
The paper~\cite{EgeKaraca-a} defines a
{\em digital fixed point property} as follows.
The digital image $(X,\kappa)$ has the digital fixed
point property with respect to the digital interval
$[0,m]_{\Z}$ if for all $(\kappa_*,\kappa)$-continuous
functions $f: (X \times [0,m]_{\Z},\kappa_*) \to (X,\kappa)$,
where $\kappa_*$ is the normal product adjacency for 
$(X,\kappa) \times ([0,m]_{\Z}, 2)$, 
there is a $\kappa$-path 
$p: [0,m]_{\Z} \to X$ of fixed points, i.e., $p(t)$ is a 
fixed point of the induced function
$f_t: X \to X$ defined by $f_t(x)=f(x,t)$. 

Also,~\cite{EgeKaraca-a} defines a
{\em digital homotopy fixed point property} and states
that this is equivalent to the following: A digital image
$(X,\kappa)$ has the digital homotopy fixed point
property if for each digital homotopy
$f: X \times [0,m]_{\Z} \to X$ there is a $\kappa$-path
$p: [0,m]_{\Z} \to X$ such that for all $t \in [0,m]_{\Z}$,
$p(t)$ is a fixed point of the induced function
$f_t$.

These imply triviality, as follows.

\begin{thm}
A digital image $(X,\kappa)$ has the digital  
fixed point property and the digital homotopy 
fixed point property if and only if $X$ is a singleton.
\end{thm}

\begin{proof}
Clearly a singleton has the digital homotopy fixed 
point property and the digital homotopy 
fixed point property. Conversely, if $X$ is not a singleton,
then by Theorem~\ref{singleton} 
there is a continuous function
$f: X \to X$ that does not have a fixed point. Let
$F: X \times [0,m]_{\Z} \to X$ be defined by $F(x,t)=f(x)$.
Then $F$ is both
$(\kappa_*,\kappa)$-continuous 
(where $\kappa_* = \kappa_*(\kappa,2)$ is the normal product 
adjacency) and a homotopy, and
fails to have a fixed point for any of the
induced functions $f_t=f$. Thus, $(X,\kappa)$ does not have
the digital fixed point property or the digital homotopy 
fixed point property.
\end{proof}

\section{Results for various contraction and expansion maps}
\subsection{Digital contraction maps}
In both of the papers \cite{EgeKaraca,Hossain-etal},  
arguments are given for the incorrect assertion
that every digital contraction map is 
digitally continuous. Both papers present an error of 
confusing (topological) continuity of a map between 
metric spaces with (digital) continuity of a map between
digital images. We present a 
counterexample to this assertion; 
our example is also used to show that Kannan,  
Chatterjea, Zamfirescu, and Rhoades contraction maps
need not be digitally continuous. We use the
Manhattan metric for its ease
of computation, but the Euclidean or other $\ell_p$ metrics could be used to obtain similar conclusions.

\begin{exl}
\label{counter}
Let 
\[ X = \{p_1 = (0,0,0,0,0), ~~ p_2 = (2,0,0,0,0), ~~ p_3 = (1,1,1,1,1)\} \subset \Z^5. \]
Let $f: X \to X$ be defined by $f(p_1)=f(p_2) = p_1$, $f(p_3)=p_2$.
Then $f$ is not $(c_5,c_5)$-continuous. However,
with respect to the Manhattan metric $d$,
$f$ is
\begin{itemize}
    \item a digital contraction map,
    \item a Kannan contraction map,
    \item a Chatterjea contraction map,
    \item a Zamfirescu contraction map,
    \item a Rhoades contraction map,
    \item a $(0.45, c_5)$-uniformly local contraction,
    \item an $\alpha-\psi$-contractive mapping for    
          $\alpha(x,y)=1$ and $\psi(t) = \lambda t$ for 
          $\lambda \in (0,1)$,
    \item a $\beta - \psi - \phi$-expansive mapping, where
          $\psi$ and $\phi$ are constant functions with the
          value 0,
    \item a weakly uniformly strict digital contraction.
\end{itemize}
\end{exl}

\begin{proof}
Note $f$ is not $(c_5,c_5)$-continuous, since
$p_1 \adj_{c_5} p_3 \adj_{c_5} p_2$, but $f(X) = \{p_1,p_2\}$ is not
$c_5$-connected.

Observe that
\[ d(p_1,p_2)=2, ~~~ d(f(p_1),f(p_2)) = 0,
\]
\[ d(p_1,p_3) = 5, ~~~ d(f(p_1),f(p_3)) = 2,\]
\[ d(p_2,p_3)= 5, ~~~ d(f(p_2),f(p_3)) = 2.\]
Therefore, we have, for all $x,y \in X$ such that $x \ne y$,
$d(f(x),f(y)) \le 2/5 \, d(x,y) < 0.45 d(x,y)$. Therefore, $f$ is a digital contraction map, a Zamfirescu contraction
map, a Rhoades contraction map, and a $(0.45, c_2)$-uniformly
local contraction.

Since
\[ d(f(x),f(y)) \le 2/5 [d(x, f(x)) + d(y, f(y))] < 
   0.45 [d(x, f(x)) + d(y, f(y))]
\]
for all $x,y \in X$ such that $x \ne y$, $f$ is a Kannan contraction map.

Note
\[ d(f(p_1),f(p_2)) = 0 < 0.45 [d(p_1,f(p_2)) + d(p_2,f(p_1))],
\]
\[ d(f(p_1,f(p_3)) = 2 < 0.45 (2 + 5) = 0.45 [d(p_1,f(p_3)) + d(p_3,f(p_1)),]
\]
\[ d(f(p_2),f(p_3)) = 2 < 0.45 (0 + 5) = 0.45 (d(p_2,f(p_3)) + d(p_3,f(p_2)).
\]
Therefore, $f$ is a Chatterjea contraction map.

That $f$ is an $\alpha-\psi$-contractive mapping for    
$\alpha(x,y)=1$ and $\psi(t) = \lambda t$ for 
$\lambda \in (0,1)$ follows from 
Remark~\ref{contractionIsAlphaPsiContractive}.

Since Example~\ref{0inPsi} notes that the constant function
with value 0 is in~$\Psi$, it follows from
Definition~\ref{expansive3} that $f$ is a
$\beta - \psi - \phi$-expansive mapping.

It follows easily from Definition~\ref{weaklyUnif} that
$f$ is a weakly uniformly strict digital contraction.
\end{proof}

The following generalizes Theorem~4.7(1) of~\cite{Han16}.
We give a proof, essentially that of~\cite{Han16}, so we 
can refer to it below.

\begin{thm}
\label{c1}
Let $(X,d,c_1)$ be a digital metric space that is $c_1$-connected,
where $d$ is any $\ell_p$ metric in $\Z^n$. Let $f: X \to X$ be a digital
contraction map. Then $f$ is a constant function.
\end{thm}

\begin{proof}
Let $\alpha \in (0,1)$ satisfy $d(f(x), f(y)) \le \alpha d(x,y)$ for all
$x,y \in X$. If $x \adj_{c_1} y$ in $X$, then $d(x,y)=1$, so
$d(f(x), f(y)) \le \alpha$, which implies $f(x)=f(y)$, since every distinct pair
of points in $\Z^n$ has distance of at least 1.

Given $x_0 \in X$,
for any $x \in X$ there is a path $P=\{x_0, x_1, \ldots, x_m = x\} \subset X$ from
$x_0$ to $x$ such that $x_i \adj_{c_1} x_{i+1}$, $0 \le i <m$. It follows from
the above that $f$ is the constant function with value $f(x_0)$.
\end{proof}

Theorem~4.7(2) of~\cite{Han16} gives examples of $c_2$-connected images
with digital contraction maps that are continuous and
not constant. However, modification
of Theorem~\ref{c1} yields the following.

\begin{thm}
\label{cu}
Let $(X,d,\kappa)$ be a digital metric space that is $\kappa$-connected, where, for some $M_1 \ge M_2 > 0$
we have that
$x \ne y$ implies $d(x,y) \ge M_2$ and
$x \adj_{\kappa}y$ implies $d(x,y) \le M_1$
Let $f: X \to X$ be a digital contraction map with
multiplier $\alpha$ such that $\alpha < M_2 / M_1$.
Then $f$ is a constant function.
\end{thm}

\begin{proof}
Let $x,y \in X$ such that $x \adj_{\kappa} y$. Then
$d(x,y) \le M_1$ and
\[d(f(x),f(y)) < \alpha d(x,y)< (M_2 / M_1) M_1 = M_2.\]
By choice of $M_2$, $f(x)=f(y)$.
It follows as in the proof of
Theorem~\ref{c1} that $f$ is constant.
\end{proof}

\begin{remark}
{\rm
Notice that Theorem~\ref{cu} applies to all connected 
digital images $(X,d,c_u)$, where $X \subset \Z^n$,
$1 \le u \le n$, and $d$ is any $\ell_p$ metric.
}
\end{remark}

\begin{remark}
{\rm
It follows from Theorem~\ref{cu} and 
Proposition~\ref{ZamfAndRhoadesContractions} that
if $(X,d,c_u)$ is a digital metric space that is $c_u$-connected,
where $X \subset \Z^n$, $1 \le u \le n$,
$d$ is any $\ell_p$ metric in $\Z^n$, and $f: X \to X$ is a 
Zamfirescu or Rhoades
contraction map such that for all $x, y \in X$,
$d(f(x),f(y)) \le \alpha d(x,y)$, where $0 < \alpha < 1/\sqrt[p]{u}$, then $f$ is a constant function. 
}
\end{remark}

\subsection{Kannan and Chatterjea contractions}
Example~\ref{counter} shows that neither a Kannan contraction map
nor a Chatterjea contraction map must be constant.
However, we have the following. 

\begin{thm}
\label{KanChatConst}
Let $(X,d,\kappa)$ be a digital metric space
of finite diameter, where $d$ is any $\ell_p$ metric.
Let $f: X \to X$ be a function. If $f$ is a Kannan contraction
map or a Chatterjea contraction map with $\alpha$ as in 
Definition~\ref{KannanDef} or Definition~\ref{ChatterjeaDef}, 
respectively, satisfying
$0 < \alpha < \frac{1}{2 \, diam X}$, then $f$ is a constant
function.
\end{thm}

\begin{proof}
We have
$d(f(x), f(y)) < 1$ for all $x,y \in X$, by
Definition~\ref{KannanDef} in the case of a Kannan contraction map,
and by Definition~\ref{ChatterjeaDef} in the case of a Chatterjea
contraction map. Since $d$ is an $\ell_p$ 
metric, it follows that $f(x)=f(y)$ for all $x,y \in X$.
\end{proof}

\subsection{Reich contractions}
\begin{thm}
\label{ReichConst}
Let $(X,d,\kappa)$ be a digital metric space of positive, finite diameter,
where $d$ is any $\ell_p$ metric.
Let $f: X \to X$ be a function. If $f$ is a Reich contraction
map with $a,b,c$ as in Definition~\ref{ReichDef} satisfying
$a,b,c \in (0, \frac{1}{3 \, diam X})$, then $f$ is a constant
function.
\end{thm}

\begin{proof}
We have
$d(f(x), f(y)) < 1$ for all $x,y \in X$. Since $d$ is an $\ell_p$ 
metric, it follows that $f(x)=f(y)$ for all $x,y \in X$.
\end{proof}

Also, it follows from Proposition~\ref{ReichIsContraction} 
that Theorem~\ref{c1} and
Theorem~\ref{cu} apply to a Reich contraction map.

\subsection{$(\alpha, \kappa)$-uniformly locally 
contractive functions}
Theorem~\ref{HossainFixed} turns out to be a trivial
result for connected digital metric spaces that use an 
$\ell_p$ metric, as we see below.

\begin{thm}
Let $(X,d,\kappa)$ be a digital metric space,
where $d$ is any $\ell_p$ metric. Let
$f: X \to X$ be an $(\alpha, \kappa)$-uniformly
locally contractive function. Then $f$ is a constant function.
\end{thm}

\begin{proof}
Let $x_0, x \in X$. Since $X$ is connected, there
is a $\kappa$-path in $X$, $\{x_i\}_{i=0}^m$, from $x_0$ to
$x$ such that $x_m=x$ and $x_i \adj_{c_1} x_{i+1}$
for $0 \le i < m$. If $d(x_i, x_{i+1}) \le 1$, then
\[ d(f(x_i),f(x_{i+1})) \le \alpha d(x_i, x_{i+1}) < 1.
\]
Since $d$ is an $\ell_p$ metric, $f(x_i) = f(x_{i+1})$.
It follows that $f$ is a constant function.
\end{proof}

\subsection{Digital expansive mappings}
We saw in Example~\ref{expansiveExl} that a digital
expansive mapping need not be digitally continuous.

Theorem~3.2 and Corollary~3.3, 
of~\cite{JR17} hypothesize 
a digital expansive mapping $T: X \to X$ that is onto.
But in ``real world" image processing, a digital image
is finite, and therefore cannot support such a map,
as shown by the following Theorems~\ref{notOnto1}
and~\ref{notOnto2}.

\begin{thm}
\label{notOnto1}
Let $(X,d,\kappa)$ be a digital metric space of more
than one point. If $X$
has a bounded diameter, then there is no self-map
$T: X \to X$ that is onto and a digital expansive mapping.
\end{thm}

\begin{proof}
Suppose there is a digital expansive mapping
$T: X \to X$. Let $x_0,y_0 \in X$ be such that 
$d(x_0,y_0) = diam X > 0$. Then for some $k >1$,
\begin{equation}
\label{expansiveEq}
     d(T(x_0),T(y_0)) \ge k d(x_0,y_0) = k\,  diam X > diam X.
\end{equation}
Since statement~(\ref{expansiveEq}) is
contradictory, the assertion follows.
\end{proof}

\begin{thm}
\label{notOnto2}
Let $(X,d,\kappa)$ be a digital metric space of more
than one point. If there exist $x_0,y_0 \in X$
such that 
\begin{equation}
\label{minEq}
d(x_0,y_0) = \min\{d(x,y) \, | \, x,y \in X, d(x,y)>0\},
\end{equation}
then there is no self-map
$T: X \to X$ that is onto and a digital expansive mapping.
\end{thm}

\begin{proof}
Suppose there exists a map $T: X \to X$ that is onto
and a digital expansive mapping. Let $x_0,y_0 \in X$ be
as in equation~(\ref{minEq}). 
Since $T$ is onto, there exist $x',y' \in X$ such that
$T(x')=x_0$, $T(y')=y_0$. Let $k$ be the expansive constant
of $T$. Then
\[ d(x_0,y_0) = d(T(x'),T(y')) \ge k d(x', y').
\]
Hence $d(x',y') \le d(x_0,y_0) / k < d(x_0,y_0)$,
which contradicts our choice of $x_0,y_0$. 
The assertion follows.
\end{proof}

Note Theorem~\ref{notOnto2} is applicable when $X$ is finite
or when $d$ is any $\ell_p$ metric.

\begin{remark}
{\rm
Example~3.8 of~\cite{JR17} claims that the self-map
$T: X \to X$ of some subset of $\Z$ given by
$T(n)=2n-1$ is onto. It is easy to see that this
claim is only true for $X = \{1\}$.
}
\end{remark}

\subsection{$\beta - \psi - \phi$-expansive
mappings}
An analog of Theorem~2.1 of~\cite{SametEtAl} is asserted as Theorem~3.2 of~\cite{JR18}:

    {\em Let $(X,d,\kappa)$ be a complete digital metric space
    and $T: X \to X$. If there exist functions
    $\beta: X \times X \to [0,\infty)$ and
    $\psi,\phi \in \Psi$ such that
    \begin{itemize}
        \item $T^{-1}$ is $\beta$-admissible;
        \item there exists $x_0 \in X$ such that
              $\beta(x_0, T^{-1}(x_0)) \ge 1$; and
        \item $T$ is digitally continuous,
    \end{itemize}
    then $T$ has a fixed point.}

However, this assertion is false, as we see in the
following.

\begin{exl}
\label{noFix}
Let $X = [-1,1]_{\Z}^2 \setminus \{(0,0)\} \subset \Z^2$.
Let $\beta = d$ be the Manhattan metric on $\Z^2$.
Let $\phi$ and $\psi$ be constant functions with value 0.
Let $f: X \to X$ be the map defined by $f(x,y)=(-x,-y)$.
Then $f$ is $\beta$-admissible; for every $p \in X$ we
have $\beta(p, f^{-1}(p)) = d(p, -p) > 1$; and $f$ is 
both $c_1$-continuous and $c_2$-continuous, but $f$ has no
fixed point.
\end{exl}

\begin{proof}
It was observed in Example~\ref{0inPsi} that the constant
function with value 0 is a member of $\Psi$. It is easy to
see that the assertion follows.
\end{proof}

The following is given as Theorem~3.3 (and, with
another hypothesis, as Theorem~3.7) of~\cite{JR18}.

\begin{thm}
\label{vacuous}
Let $(X,d,\kappa)$ be a complete digital metric space and
let $T: X \to X$ be $\beta - \psi - \phi$-expansive
mapping such that for some sequence
$\{x_n\}_{n=1}^{\infty} \in X$ we have
$\beta(x_n,x_{n+1}) \ge 1$ for all $n$ and
$x_n \rightarrow y \in X$ as $n \rightarrow \infty$, then
there is a subsequence $\{x_{n_k}\}$ of
$\{x_n\}_{n=1}^{\infty}$ such that
$\beta(x_{n_k},y) \ge 1$ for all $k$.
Then $T$ has a fixed point.
\end{thm}

However, we have the following.

\begin{exl}
If $X$ is finite or $d$ is an $\ell_p$ 
metric, and $\beta = d$, then
Theorem~\ref{vacuous} is vacuously true, as no such
sequence $\{x_n\}_{n=1}^{\infty} \in X$ exists.
\end{exl}

\begin{proof}
By Corollary~\ref{Han-Cauchy-cor}, $x_n \rightarrow y$
implies that for some $k_0$, $k>k_0$ implies
$x_n=y$ and therefore 
$\beta(x_n,y)=d(x_n,y)=0$. Thus, no sequence
$\{x_n\}_{n=1}^{\infty}$ satisfies the hypotheses of
Theorem~\ref{vacuous}.
\end{proof}

\begin{remark}
{\rm
The following assertion is stated as Theorem~3.6
of~\cite{JR18}.}

Let $(X,d,\kappa)$ be a complete digital metric
space. Let $T: X \to X$ be a $\beta - \psi - \phi$-expansive
mapping such that
\[ \psi(d(T(x),T(y))) \ge \beta(x,y) \psi(M(x,y)) +
                          \phi(M(x,y))
\]
for all $x,y \in X$, where
\[ M(x,y) = \max \{d(x,y), d(x,T(x)), d(y,T(y)),
                   \frac{d(x,T(y)) + d(y, T(x))}{2}\}.
\]
Then $T$ has a fixed point.

{\rm But this assertion is false.
}
\end{remark}

\begin{proof}
It is easy to see that the choices of $X,d,\kappa,T,\beta,\psi,\phi$ of
Example~\ref{noFix} provide a counterexample to the assertion.
\end{proof}

\subsection{Remarks on~\cite{Dol-Nal}}
The paper~\cite{Dol-Nal} identifies one of the authors
of the current paper, L. Boxer, as a reviewer. In fact, errors and other
shortcomings mentioned in Boxer's review 
remain in the published version of~\cite{Dol-Nal}.

The assertion stated as Theorem~3.1
of~\cite{Dol-Nal} is the following.
\begin{quote}
    {\em Let $(X,d,\kappa)$ be a complete
         metric space such that $T: X \to X$
         satisfies $d(T(x),T(y)) \le \psi(d(x,y))$ for all $x,y \in X$,
         where $\psi: [0,\infty) \to [0,\infty)$ is monotone
         nondecreasing and
         $\psi^n(t) \rightarrow 0$ as
         $n \rightarrow \infty$. Then $T$
         has a unique fixed point.
    }
\end{quote}

The argument offered in proof of this assertion confuses topological
continuity (the ``$\varepsilon - \delta$ definition")
and digital continuity (preservation of connectedness) in order to conclude that
$T$ is continuous. However, in
Example~\ref{counter}, using $\psi(t)=t/2$,
we have a function that satisfies the
hypotheses above and is not digitally
continuous.

A similar flaw appears in the argument given
in proof of the assertion states as
Theorem~3.3 of~\cite{Dol-Nal}, where again
Example~\ref{counter} provides a counterexample to the claim that a
weakly uniformly strict digital contraction mapping is digitally continuous.

Therefore, we must regard the assertions stated as Theorems~3.1 and~3.3
of~\cite{Dol-Nal}
as unproven. Since these and assertions dependent
on these make up all of the new assertions of the paper,
we conclude that nothing new is correctly 
established in~\cite{Dol-Nal}.

\section{Common fixed points of intimate maps}
The paper~\cite{Jain} obtains a result
for common fixed points of intimate maps.
We show in this section that the
characterization of intimate maps given 
in~\cite{Jain} can be simplified, and that
the primary result of~\cite{Jain} is
rather limited.

\begin{definition}
{\rm \cite{Jain}}
\label{intimateDef}
Let $(X,d,\kappa)$ be a digital metric
space. Let $f,g: X \to X$. Let $\alpha$
be either the $lim \, inf$ or the
$lim \, sup$ operation. If for
every $\{x_n\}_{n=1}^{\infty} \subset X$ 
such that
\begin{equation}
\label{intimateLim}
\lim_{n \to \infty} f(x_n) =
   \lim_{n \to \infty} g(x_n) = t
\end{equation}
for some $t \in X$ we have,
for $n$ sufficiently large,
\begin{equation}
\label{intimateEq}
     \alpha \, d(g(f(x_n)), g(x_n)) \le
   \alpha \, d(f(f(x_n)), f(x_n))
\end{equation}
then we say $f$ is {\em $g$-intimate}.
\end{definition}

\begin{prop}
Let $(X,d,\kappa)$ be a digital metric
space, where $d$ is any $\ell_p$ metric.
Let $f,g: X \to X$. Then $f$ is
$g$-intimate if and only if for every
sequence $\{x_n\}_{n=1}^{\infty} \subset X$
satisfying statement~(\ref{intimateLim})
we have, for $n$ sufficiently large,
\[ d(g(t), t) \le d(f(t),t).
\]
\end{prop}

\begin{proof}
From Theorem~\ref{Han-Cauchy}, a
sequence $\{x_n\}_{n=1}^{\infty} \subset X$
satisfying statement~(\ref{intimateLim}) has,
for $n$ sufficiently large,
$f(x_n)=g(x_n)=t$. The assertion follows
easily.
\end{proof}

\begin{thm}
{\rm ~\cite{Jain}}
\label{jainThm}
If $(X,d,\kappa)$ is a digital
    metric space and
    $A,B,S,T: X \to X$ are such that
    
    a) $S(X) \subset B(X)$ and
              $T(X) \subset A(X)$;
              
    b) for some $\alpha \in (0,1)$ and
              all $x,y \in X$, 
    \begin{equation}
    \label{STFineq}
    d(S(x),T(y)) \le \alpha \, F(x,y)
    \end{equation}
    where
\[  F(x,y) =
   \max \{d(A(x),B(y)), d(A(x),S(x)),
   d(B(y),T(y)),\]
\[   ~~~d(S(x),B(y)), d(A(x),T(y))\};
\]

c) $A(X)$ is complete;

d) $S$ is $A$-intimate and
          $T$ is $B$-intimate,

\noindent  then $A$, $B$, $S$, and $T$ have a
    unique common fixed point.
\end{thm}

But Theorem~\ref{jainThm} is limited, as shown
by the following.

\begin{prop}
Suppose we assume the hypotheses of
Theorem~\ref{jainThm}, with $d$ being
an $\ell_p$ metric. Suppose
\begin{equation}
\label{alphaBound} 
\alpha < \min \{1 / F(x,y) \, | \, x,y \in X \}
\end{equation}
or
\begin{equation}
\label{diamBound} 
diam(S(X) \cup T(X)) = diam \, X < \infty.
\end{equation}
Then $S$ and $T$ are constant functions that 
have the same value, and, in
case~(\ref{diamBound}), $X$ is a singleton.
\end{prop}

\begin{proof}
Inequality~(\ref{alphaBound}) and 
hypothesis b) of Theorem~\ref{jainThm}
imply that $d(S(x),T(y)) < 1$, hence
$S(x)=T(y)$ for all $x, y \in X$.
Since $d$ is an $\ell_p$
metric, we have, for some $x_0, y_0 \in X$,
$S(x_0)=T(y)$ and $S(x)=T(y_0)$ for all
$x, y \in X$. Hence $S$ and $T$ are constant
functions with the same values.

Statement~(\ref{diamBound}) implies there
exist $x_0,y_0 \in X$ such that
$d(S(x_0),T(y_0)) = diam \, X$. Since
$F(x_0,y_0) \le diam \, X$,
inequality~(\ref{STFineq}) becomes
$diam \, X \le \alpha \, diam \, X$, which
implies $diam \, X = 0$. Thus, $X$ is a
singleton, so $S$ and $T$ are constant
functions with the same values.
\end{proof}

\section{Homotopy invariant fixed point theory}
It is natural in topology to consider the behavior of the fixed point set when a continuous function is changed by homotopy. In classical topology for nice spaces (for example the geometric realization of any finite simplicial complex), it is trivial to change a function by homotopy to increase the number of fixed points. The more interesting question is whether or not the number of fixed points can be decreased by homotopy.

The following Proposition~\ref{noFixed} was the key to
the proof of Theorem~\ref{singleton}.

\begin{prop}
{\rm \cite{Bx-etal}}
\label{noFixed}
Let $(X,\kappa)$ be a connected digital image of more than
one point. Let $x_0 \adj_{\kappa} x_1$ in $X$. Then the function $g: X \to X$ defined by
\[g(x) = \left \{ \begin{array}{cc}
    x_0 & \mbox{if } x \ne x_0; \\
     x_1 & \mbox{if } x = x_0,
\end{array} \right .
\]
is continuous and has no fixed points.
\end{prop}

\begin{prop}
\label{htpNoFixed}
Let $(X,\kappa)$ be a connected digital image of more than
one point. Then any constant map $f: X \to X$ is homotopic
to a map without fixed points.
\end{prop}

\begin{proof}
Let $x_0, x_1 \in X$ with $x_0 \adj_{\kappa} x_1$. Let
$f: X \to X$ be the constant map with image $\{x_0\}$.
Let $H: X \times [0,1]_{\Z} \to X$ be defined by
$H(x,0)= x_0$; $H(x,1)=x_0$ for $x \ne x_0$; $H(x_0,1)=x_1$.
It is easy to see that $H$ is a homotopy from $f$ to a
function $g$ as in Proposition~\ref{noFixed} without fixed points.
\end{proof}

Let $\MF(f)$ be the minimal number of fixed points among all continuous functions homotopic to $f$. For example, if $X$ has only 1 point then clearly $\MF(f) = 1$. If $X$ has more than 1 point and $(X,\kappa)$ is contractible, then any continuous function $f:X\to X$ is homotopic to a constant 
function, and it follows from Proposition~\ref{htpNoFixed} that $\MF(f) = 0$. 

Several examples are given in \cite{hmps} of digital images $(X,\kappa)$ for which no function on $X$ is homotopic to the identity except for the identity itself. Such images
are called \emph{rigid}. For example, a wedge product of two loops, each having at least 5 points, is rigid. Clearly if $X$ is a rigid digital image having $n$ points and $\id$ denotes the identity function, then $\MF(\id) = |X|$.

In classical topology, $\MF(f)$ can often be computed by Nielsen fixed point theory; see \cite{jiang}. Each fixed point is assigned an integer-valued \emph{fixed point index}, which can be computed homologically. 
When $x$ is an isolated fixed point of $f$, the fixed point index of $x$ is denoted $\ind(f,x)$. 
This index is a sort of multiplicity count for the fixed point: when $\ind(f,x)= 0$ then the fixed point at $x$ can be removed by a homotopy. 

The fixed point index is homotopy invariant in the following sense: if, during some homotopy $f\simeq g$, the fixed point $x$ of $f$ moves into a fixed point $y$ of $g$, then $\ind(f,x) = \ind(g,x)$. Furthermore, when the fixed point set of $f$ is finite, then the sum of all the fixed point indices equals the Lefschetz number $L(f)$. In this sense $\ind(f,x)$ provides a localized version of the Lefschetz number. 

In Nielsen fixed point theory, the fixed points are grouped into \emph{Nielsen classes}, and the number of such classes having nonzero index sum is the \emph{Nielsen number} $N(f)$. This number is a homotopy invariant satisfying $N(f) \le \MF(f)$, and in many cases (for example when $X$ is a manifold of dimension different from 2), $N(f) = \MF(f)$.

The 2012 paper~\cite{eklefschetz} by Ege \& Karaca attempts to develop a Lefschetz fixed point theorem for digital images, but the main result is incorrect, and was retracted in the 2016 paper \cite{Bx-etal}. The same authors attempted to develop a Nielsen theory in the 2017 paper \cite{EgeKaraca17} based on their faulty Lefschetz theory. The theory developed in \cite{EgeKaraca17} is also incorrect.

The main problem in \cite{EgeKaraca17} is inherited from problems in \cite{eklefschetz}, and concerns the definition of the fixed point index. Definition 3.2 of \cite{EgeKaraca17} does not give a satisfactory definition of the function $F$, and the degree used is inadequate because the appropriate homology groups are not necessarily $\Z$. The authors claim to define an integer valued Nielsen number $N(f)$ which is a homotopy invariant (Theorem 3.6) and a lower bound for $\MF(f)$ (Theorem 3.7). 
Their Example 3.4, claiming that if $f$ is a 
constant then $N(f)=1$, yields a contradiction for connected images $X$ such that $|X|>1$, since our Proposition~\ref{htpNoFixed} implies that $MF(f)=0$ for such functions $f$.

The errors in this work are not merely mistakes but indicate fundamental flaws in the theory. Anything resembling the standard homological definitions of $L(f)$ and the fixed point index will require that the Lefschetz number and fixed point index of the constant map equal 1. This cannot be reconciled with the fact that, when $X$ has more than 1 point, the constant map can be changed by homotopy to have no fixed points.

The authors believe that any successful theory for computing $\MF(f)$ will involve techniques very different from classical Lefschetz and Nielsen theory. The setting of digital images also allows the study of the quantity $\XF(f)$, the maximum number of fixed points among all functions homotopic to $f$. In classical topological fixed point theory this number is typically infinite, but for a digital image with $n$ points, clearly $\XF(f) \le n$ for any $f$. In fact our definition of $\XF(f)$ implies that $f$ is homotopic to the identity if and only if $\XF(f)=n$. When $f$ is a constant, then $1 \le \XF(f)\le n$, and for many choices of the image $X$ we will have $\XF(f) < n$. We do not know if it is possible for $\XF(f)=0$ for any function on a connected digital image.

Although many of the concepts discussed in this paper turn out to be trivial or otherwise uninteresting, the questions of computing $\MF(f)$ and $\XF(f)$ seem to be difficult and interesting, and present opportunities for further work. 
Variations that count approximate fixed points would also be interesting objects of study.

\section{Concluding remarks}
Although the study of fixed points, or approximate fixed
points, is important in digital topology as in other
branches of mathematics, it does not appear that the use
of metric spaces yields useful knowledge in this area. We
have seen that metric space functions introduced to 
study fixed points in digital topology - digital contraction
maps, Kannan contraction maps, Chatterjea contraction maps, 
Zamfirescu contraction maps, Rhoades contraction maps,
Reich contraction maps, uniformly locally 
contractive functions, intimate functions
- often turn out to be either discontinuous 
or constant - hence, arguably uninteresting - 
when the image considered is finite or when 
common metrics are used.

It appears to us that the most natural metric function
to use for a connected digital image $(X,\kappa)$ is
the path length metric~\cite{Han05}:
$d(x,y)$ is the length of a shortest $\kappa$-path from
$x$ to $y$. Since this metric reflects
$\kappa$, it seems far superior to
an $\ell_p$ metric on a digital image. However, even this
metric gives us little new information. Since it is
integer-valued, its Cauchy sequences are also eventually
constant.

We have also corrected errors and pointed out trivialities  
in other papers concerned with fixed points or approximate 
fixed points of continuous self-maps of digital images.


\begin{thebibliography}{11}

\bibitem{Banach}
S. Banach, Sur les operations dans les ensembles abstraits et leurs applications aux equations integrales,
{\em Fundamenta Mathematicae} 3 (1922), 133-181.

\bibitem{Berge}
C. Berge, {\em Graphs and Hypergraphs}, 2nd edition, 
North-Holland, Amsterdam, 1976.

\bibitem{Bx94}
L. Boxer, Digitally continuous functions,
{\em Pattern Recognition Letters} 15 (1994), 833-839.

https://www.sciencedirect.com/science/article/pii/0167865594900124?via

\bibitem{Bx99}
L. Boxer, A classical construction for the digital fundamental group,
{\em Journal of Mathematical Imaging and Vision} 10 (1999), 51-62.

https://link.springer.com/article/10.1023/A

\bibitem{Bx17}
L. Boxer,
Generalized normal product adjacency in digital topology, 
{\em Applied General Topology} 18 (2) (2017), 401-427.

https://polipapers.upv.es/index.php/AGT/article/view/7798/8718

\bibitem{Bx18}
L. Boxer, Alternate product adjacencies in digital 
topology, 
{\em Applied General Topology} 19 (1) (2018), 21-53

https://polipapers.upv.es/index.php/AGT/article/view/7146

\bibitem{Bx-etal}
L. Boxer, O. Ege, I. Karaca, J. Lopez, and J. Louwsma, 
Digital fixed points, approximate fixed points, and universal functions,
{\em Applied General Topology} 17(2) (2016), 159-172.

https://polipapers.upv.es/index.php/AGT/article/view/4704


\bibitem{Chatterjea}
S.K. Chatterjea, Fixed point theorems,
{\em Comptes rendus de l'Académie bulgare des Sciences} 25 (1972), 727-730.

\bibitem{Dol-Nal}
U.P. Dolhare and V.V. Nalawade,
Fixed point theorems in digital images and applications
to fractal image compression,
{\em Asian Journal of Mathematics and Computer Research}
25 (1) (2018), 18-37. 

http://www.ikpress.org/abstract/6915

\bibitem{Edelstein}
M. Edelstein,
An extension of Banach's contraction principle,
{\em Proceedings of the American Mathematical Society}
12 (1) (1961), 7-10. 

http://www.ams.org/journals/proc/1961-012-01/S0002-9939-1961-0120625-6/S0002-9939-1961-0120625-6.pdf 

\bibitem{eklefschetz}
O. Ege and I. Karaca, The Lefschetz Fixed Point Theorem for Digital Images, {\em Fixed Point Theory and Applications} 2013:253 2013

https://fixedpointtheoryandapplications.springeropen.com/track/pdf/10.1186/1687-1812-2013-253

\bibitem{EgeKaraca}
O. Ege and I. Karaca, Banach fixed point theorem for digital images,
{\em Journal of Nonlinear Sciences and Applications}, 
8 (2015), 237-245.

http://www.isr-publications.com/jnsa/articles-1797-banach-fixed-point-theorem-for-digital-images

\bibitem{EgeKaraca-a}
O. Ege and I. Karaca,
Digital homotopy fixed point theory,
{\em Comptes Rendus Mathematique} 353 (11) (2015), 1029-1033.

https://www.sciencedirect.com/science/article/pii/S1631073X15001909

\bibitem{EgeKaraca17}
O. Ege and I. Karaca,
Nielsen fixed point theory for digital images,
{\em Journal of Computational Analysis and Applications} 22 (5) (2017), 874-880.

\bibitem{hmps}
J. Haarmann, M.P. Murphy, C.S. Peters, and P.C. Staecker,
Homotopy equivalence of finite digital images,
{\em Journal of Mathematical Imaging and Vision} 53, (3), (2015), 288-302.

https://link.springer.com/article/10.1007/s10851-015-0578-8

\bibitem{Herman}
G. Herman, Oriented surfaces in digital spaces,
{\em CVGIP: Graphical Models and Image Processing}
55 (1993), 381-396.

https://www.sciencedirect.com/science/article/pii/S1049965283710291

\bibitem{Han05}
S-E Han,
Non-product property of the digital fundamental group,
{\em Information Sciences} 171 (2005), 73–91.

https://www.sciencedirect.com/science/article/pii/S0020025504001008


\bibitem{Han16} 
S-E Han,
Banach fixed point theorem from the viewpoint of digital topology,
{\em Journal of Nonlinear Science and Applications} 9 (2016), 895-905.

http://www.isr-publications.com/jnsa/articles-1915-banach-fixed-point-theorem-from-the-viewpoint-of-digital-topology



\bibitem{Han17} 
S-E Han,
The fixed point property of an M-retract and its applications,
{\em Topology and its Applications} 230, 139-153.


\bibitem{Hossain-etal}
A. Hossain, R. Ferdausi, S. Mondal, and H. Rashid,
Banach and Edelstein fixed point theorems for digital images,
{\em Journal of Mathematical Sciences and Applications} 5 
(2) (2017), 36-39.

http://www.sciepub.com/jmsa/abstract/8331

\bibitem{Jain}
D. Jain,
Common fixed point theorem for intimate
mappings in digital metric spaces,
{\em International Journal of Mathematics Trends and Technology} 56 (2) (2018),
91-94.

http://www.ijmttjournal.org/archive/ijmtt-v56p511

\bibitem{jiang}
B. Jiang, Lectures on Nielsen fixed point theory, Contemporary Mathematics 18 (1983).

https://bookstore.ams.org/conm-14

\bibitem{JR17}
K. Jyoti and A. Rani,
Digital expansions endowed with fixed point theory,
{\em Turkish Journal of Analysis and Number Theory} 5 (5) 
(2017), 146-152.

http://pubs.sciepub.com/tjant/5/5/1/index.html

\bibitem{JR18}
K. Jyoti and A. Rani,
Fixed point theorems for $\beta - \psi - \phi$-expansive
type mappings in digital metric spaces,
{\em Asian Journal of Mathematics and Computer Research}
24 (2) (2018), 56-66.

http://www.ikpress.org/abstract/6855

\bibitem{Kannan}
R. Kannan, Some results on fixed points,
{\em Bulletin of the Calcutta Mathematical Society} 60 (1968), 71-76.

\bibitem{Mishra-etal}
L.N. Mishra, K. Jyoti, A. Rani, and Vandana,
Fixed point theorems with digital contractions image processing,
{\em Nonlinear Science Letters A} 9 (2) (2018), 104-115.

http://www.nonlinearscience.com/paper.php?pid=0000000271

\bibitem{Ege-etal}
C. Park, O. Ege, S. Kumar, D. Jain, and J. R. Lee,
Fixed point theorems for various contraction conditions in digital metric spaces,
{\em Journal of Computational Analysis and Applications} 26 (8) (2019), 1451-1458.

\bibitem{Reich}
S. Reich, Some remarks concerning contraction mappings,
{\em Canadian Mathematical Bulletin}, 14 (1971), 121-124.

https://cms.math.ca/10.4153/CMB-1971-024-9

\bibitem{Rhoades}
B.E. Rhoades,
Fixed point theorems and stability results for
fixed point iteration procedures, II,
{\em Indian Journal of Pure and Applied Mathematics}
24 (11) (1993), 691-703.

\bibitem{Rosenfeld87}
A. Rosenfeld, `Continuous' functions on digital images,
{\em Pattern Recognition Letters} 4 (1986), 177-184.

https://www.sciencedirect.com/science/article/pii/0167865586900176?via

\bibitem{SametEtAl}
B. Samet, C. Vetro, and P. Vetro,
Fixed point theorems for $\alpha - \psi$-contractive mappings,
{\em Nonlinear Analysis: Theory, Methods \& Applications}
75 (4) (2012), 2154-2165.

https://www.sciencedirect.com/science/article/pii/S0362546X1100705X

\bibitem{Zamf72}
T. Zamfirescu,
Fixed point theorems in metric spaces,
{\em Archiv der Mathematik} 23 (1972), 292-298.

https://link.springer.com/article/10.1007/BF01304884

\end{thebibliography}
\end{document}